\UseRawInputEncoding

\documentclass[12pt]{amsart}
\usepackage{epsfig,color}
\usepackage{esint}
\usepackage{hyperref}
\usepackage{mathrsfs}

  \hypersetup{colorlinks}
\hypersetup{citecolor=blue}
\hypersetup{urlcolor=blue}

\makeatother

\theoremstyle{definition}

\def\fnum{equation} 
\newtheorem{Thm}[\fnum]{Theorem}
\newtheorem{Cor}[\fnum]{Corollary}

\newtheorem{Lem}[\fnum]{Lemma}

\newtheorem{Def}[\fnum]{Definition}

\newtheorem{Rem}[\fnum]{Remark}
\newtheorem{Pro}[\fnum]{Proposition}

\numberwithin{equation}{section}

\newcommand{\supp}{{\text{supp}}}

 \newcommand{\N}{\ensuremath{\mathbb{N}}}
 
  \newcommand{\R}{\ensuremath{\mathbb{R}}}

 \newcommand{\ba}{\begin{align*}}
 \newcommand{\ea}{\end{align*}}
 \newcommand{\na}{\nabla}
\newcommand{\la}{\langle}
\newcommand{\ra}{\rangle}

\newcommand{\ep}{\epsilon}

\DeclareMathOperator{\RCD}{RCD}
\DeclareMathOperator{\CD}{CD}
\DeclareMathOperator{\MCP}{MCP}

\DeclareMathOperator{\Ric}{Ric}

\title{Failure of strong unique continuation for harmonic functions on RCD Spaces}
\author{Qin Deng}%
\author{Xinrui Zhao}%
\address{MIT, Dept. of Math.\\
77 Massachusetts Avenue, Cambridge, MA 02139-4307.}
\thanks{}

\email{}

\begin{document}

\maketitle

 \begin{abstract}
 Unique continuation of harmonic functions on $\RCD$ space is a long-standing open problem, with little known even in the setting of Alexandrov spaces. In this paper, we establish the weak unique continuation theorem for harmonic functions on $\RCD(K,2)$ spaces and give a counterexample for strong unique continuation in the setting of $
 \RCD(K,N)$ space for any $N\geq 4$ and any $K\in \R$. 
  \end{abstract}\section{introduction}

Unique continuation of solutions to second order elliptic equations on Euclidean domains is a  well-studied problem in PDE theory. Carleman's pioneering work \cite{Ca} established strong unique continuation in dimension 2 with $C^2$ coefficients. Later, Bers and Nirenberg \cite{BN} showed the same for equations in non-divergence form with measurable coefficients using tools from complex analysis. For equations in divergence form, some results were obtained by \cite{AM, Al, Sc}.
 
When $n\geq 3$, unique continuation was demonstrated by Aronszajn \cite{Ar} for equations with $C^{2,1}$ coefficients and later improved to equations with $C^{0,1}$ Coefficients by Aronszajn, Krzywicki and Szarski \cite{AKS}  using a Carleman-type inequality. Garofalo and Lin \cite{GL1, GL2} used a simpler argument involving frequency estimates to show a similar result. The regularity is sharp by Plis \cite{Pl} and Miller \cite{M}. For more discussions on the frequency function and their applications, see for example \cite{CM1, CM2, Alm, CM3, CM4}.

 In this note we will prove two results -- the failure of strong unique continuation in $\RCD(K,N)$ spaces for $N \geq 2$ and weak unique continuation in the case of $\RCD(K,2)$ spaces.
 
Consider the equation\begin{align}
    Lu = -D_\alpha(ag^{\alpha\beta}D_\beta u+b^\alpha u)+c^\alpha D_\alpha u+eu=0
    \label{equm}
\end{align} with  \begin{align}
    a\in C^{0,1},\,\,a\geq \gamma>0,\,\, b^\alpha,c^\alpha,e\in L^\infty,\,\,g^{\alpha\beta}\,\,\text{is the metric tensor}.
    \label{condm}
    \end{align} Our first result is unique continuation for solutions to such equations in the two-dimensional case.
  \begin{Thm}
 Given function $\phi\in W^{1,2}(\Omega)$ which is a solution of \eqref{equm} on a connected open subset $\Omega$ of a $\RCD(K,2)$ space $(X,d,m)$, if we have that $\phi|_{B_r(x)}=0$ for some $x\in \Omega$ and $r>0$, then $\phi\equiv 0$.
 \label{thm:c01}
 \end{Thm}
We note that by a result of Lytchak and Stadler \cite{LS}, it suffices to consider Alexandrov space in dimension 2 of curvature at least K and collapsed $\RCD(K,2)$ spaces, which is classified in \cite{KL}.

 It is known by a result of Nikolaev \cite{Ni} (cf. Berestovskij and Nikolaev \cite{BeN}) that Alexandrov spaces with sectional curvature bounded both from above and below have local harmonic coordinates around every point such that the metric tensor is $C^{1,\alpha}$. Combine with results of Aronszajn, Krzywicki and Szarski \cite{AKS} and Garofalo and Lin \cite{GL1}, one obtains  strong unique continuation.  
 
 More precisely, we say that u vanishes to infinite order at $x_0$ if there exists $R>0$ such that for each integer $N>0$ 
 \begin{align}
     \int_{B_r(x_0)}|u|^2dx\leq c_Nr^N,\quad \text{ for }r<R.
     \label{vio}
 \end{align}
 \begin{Thm}
   Given function $\phi\in W^{1,2}(\Omega)$ which is a solution of \eqref{equm} on a connected open subset $\Omega$ of an Alexandrov space $(X,d)$ with curvature bounded both from above and below, if we have that $\phi$ vanishes of infinite order to $x_0$ (see \eqref{vio}) for some $x_0\in \Omega$ , then $\phi\equiv 0$.
  \label{pro:c02}
 \end{Thm}

We also discuss spaces with two-sided bounded Ricci curvature or with Ricci curvature bounded from below and sectional curvature bounded from above in later sections. Finally we show that if we only assume that $(X,d,m)$ is $\RCD(K,N)$, there are counterexamples for strong unique continuation.

\begin{Thm}
For each $N \geq 4$ and $K$, there exists some $\RCD(K,N)$ space $(X,d,m)$ and some $u$ which is harmonic on some connected $\Omega \subseteq X$ and vanishing to infinite order at some $x_0 \in \Omega$, but that is not identically zero. Moreover, if $K\leq 0$, $u$ may be chosen to be a global harmonic function.
\end{Thm}
\begin{Rem}
 Actually we can prove that there exists a point $x_0$ in an  $\RCD(K,N)$ space $(X,d,m)$ such that every harmonic function which vanishes at $x_0$, vanishes at infinite order.
\end{Rem}
We note that we were only able to construct counterexamples when the point $x_0$ in the above proposition is singular, and it remains an open question whether this is possible at some regular point.
 

 \section*{Acknowledgement}
We are very grateful to Prof.\,\,Tobias Colding for introducing us to the unique continuation problem and several inspirational discussions. We thank Prof.\,\,Vitali Kapovitch for helpful discussions and comments during the writing of this paper. We also thank the anonymous referees for comments and suggestions. Xinrui Zhao is supported by NSF Grant DMS 1812142.
 
 \section{Local Charts in the 2-dimensional case}
 For an $\RCD(K,2)$ space $(X,d,m)$, we have, from Lytchak and Stadler \cite{LS}, that
 \begin{Pro}(\cite{LS})
If $m = \mathcal{H}^2$ (i.e. the space is non-collapsed), then $(X,d)$ is an Alexandrov space with curvature at least $K$.
 \label{pro:raequ}
 \end{Pro}
 
 Therefore, it suffices to consider $2$-dimensional Alexandrov spaces of curvature at least $K$ and collapsed $\RCD(K,2)$ spaces. For an Alexandrov space, we have that 
 \begin{Pro}(\cite{OS})
 Let $(X,d)$ be an $n$-dimensional Alexandrov space and denote $S_X$ as the set of singular points. Then there exists a $C^0$-Riemannian structure on $X\setminus S_X\subset X$ satisfying the following :
\begin{itemize}
    \item[(1)]There exists an $X_0\subset  X\setminus S_X$ such that $X\setminus  X_0$ is of $n$-dimensional Hausdorff measure zero and that the Riemannian structure is $C^{\frac{1}{2}}$-continuous on $X_0\subset X$;
    \item[(2)]The metric structure on $X\setminus S_X$ induced from the Riemannian structure coincides with the original metric of $X$.
\end{itemize}
 \label{pro:c1}
 \end{Pro}
 
 See \cite[Section 1]{OS} for a definition of $C^r$-Riemannian structure. If an Alexandrov space has sectional curvature not only bounded from below, but also from above, we have the stronger
 \begin{Pro}(\cite{Ni}, cf. \cite{BeN,KKK})
 Let $(X,d)$ be an Alexandrov space with bounded curvature. Then in a neighbourhood of each interior point there exists a harmonic coordinate system so that the components of the metric tensor in any harmonic coordinate system in $X$ are continuous functions of class $W^{2,p}$, where for $p$ we can take any number not less than one. Harmonic coordinate systems specify on $X$ an atlas of
class $C^{3,\alpha}$ for any $0<\alpha <1$.
 \label{pro:c2}
 \end{Pro}
 
 From the Sobolev embedding theorem for Alexandrov space (cf. \cite{LU}), this implies that the metric tensor is $C^{1,\alpha}$ in the corresponding harmonic coordinate systems. 
 
 Note that the major difference between Proposition \ref{pro:c1} and Proposition \ref{pro:c2} is that once we have bounded curvature, a good chart can be found around any point in the space, while if curvature is bounded merely from below, the chart can only be found around regular points and the metric tensor might even be discontinuous on the chart, if we do not remove an n-dimensional Hausdorff measure zero set, see for example \cite{AB}. We will handle this difficulty in the following sections.

Recall that a solution (sub-solution, super-solution) of \eqref{equm} on an Alexandrov space is defined as (cf. \cite{OS, Sh, KMS, H, AGS, GiT}) 
 \begin{Def}
 A function $u: X\to \R$ is a solution (resp. sub-solution, super-solution) of \eqref{equm} if $u\in W^{1,2}_{loc} (X)$ such that for all $\phi\in Lip_0(X)$ and $\phi\geq 0$, we have \begin{align}
     \int_X  -a\na u\cdot \na \phi- ub\cdot\na \phi+\phi c\cdot\na u+e(z)u\phi\,\, dm=0\, (\text{resp. }\leq 0, \geq 0).
     \label{requ}
 \end{align}
 where the inner product is induced by the Riemannian structure around regular points.
 \end{Def}
 
Using local charts, we may consider harmonic functions locally around regular points  as a solution of an elliptic equation in divergence form with measurable coefficients\begin{align}
     \int_X  -g_{ij}u_i\phi_j- ub_i\phi_i+\phi c_i u_i+e(z)u\phi\,\, dm=0\, (\text{resp. }\leq 0, \geq 0).
     \label{requ2}
 \end{align}

\section{unique continuation in Alexandrov space with bounded curvature}
In general, consider on a Euclidean domain the equation \begin{align}
    Lu = -D_\alpha(a^{\alpha\beta}(z)D_\beta u+b^\alpha(z)u)+c^\alpha(z)D_\alpha u+e(z)u=0
    \label{equ}
\end{align} with  \begin{align}
    a^{\alpha\beta}\in C^{0,1},\,\,\lambda\, I\leq (a^{\alpha\beta})\leq \Lambda\, I,\,\, b,c,e\in L^\infty.
    \label{cond}
    \end{align} 
    
    It is known that
    
    \begin{Pro}(\cite{AKS,GL2}, cf. \cite{KT})
 If u is a solution of \eqref{equ} with coefficients satisfying \eqref{cond} and vanish of infinite order at an interior point $x_0$ (see \eqref{vio}), then we know that $u=0$ in the connected component of the domain containing $x_0$.
 \label{gensuc}
 \end{Pro}
    For weaker assumptions on coefficients, see for example Koch and Tataru \cite{KT}. Using the previous proposition, we have
    
  \begin{Thm}
  If $\phi\in W^{1,2}(\Omega)$ is a solution of \eqref{equm} on a connected open subset $\Omega$ of an Alexandrov space $(X,d)$ with curvature bounded both from above and below, and if there exists $x_0\in \mathring{\Omega}$ and $R>0$ such that for each integer $N>0$  we have \eqref{vio}, then $\phi\equiv 0$.
  \label{Thm:c02}
 \end{Thm}
 \begin{proof}
 Let us consider a harmonic coordinate system around $x_0$. From Proposition \ref{pro:c2} we know that the metric tensor $g^{\alpha\beta}$ are $C^{1,\alpha}$. Then we can rewrite the equation \eqref{requ} as
  \begin{align}
      \int_{\R^n}  (a\,g^{\alpha\beta}D_\alpha u  D_\beta \phi+ u\,b^\alpha D_\alpha\phi+\phi\,c^\alpha D_\alpha u+e(z)\,u\,\phi)\, \sqrt{det(g_{ij})}dx=0, \text{ for all } \phi\in Lip_0(U).
  \end{align}
  As we have $a\sqrt{det(g_{ij})}g^{\alpha\beta}\in C^{0,1}$, $\lambda\,I\leq (a\sqrt{det(g_{ij})}g^{\alpha\beta})\leq \Lambda \,I$ and the rest coefficients are all $L^\infty$. So from Proposition \ref{gensuc} we have that $u$ is zero around $x_0$. This shows that the points at which $u$ vanishes of infinite order form an open set $V$.
  
  Now it suffices to show that the set V of points at which u vanishes of infinite order is also a closed set. Given any $x\in \overline{V}\setminus V$, there is a sequence of points $x_i\in V$ with $x_i\to x$. From Proposition \ref{pro:c2} we know that around any point in the space there is a harmonic coordinate chart. So consider a harmonic chart C around x. There exists $N$ such that for $i\geq N$, we have $x_i\in C\cap V$. From Proposition \ref{gensuc} we know that $\phi=0$ on C. This shows that $x$ is also a point at which u vanishes of infinite order. So $\overline{V}=V$ and $V$ is open. As $\Omega$ is connected and V is non-empty, we know that $V=\Omega$, that is, $\phi=0$ on $\Omega$.

 \end{proof}

 \section{Unique continuation in non-collapsed $\RCD(K,N)$ space}
 It is natural to ask if it would be possible to weaken the assumption of bounded sectional curvature. A natural generalization would be to weaken the assumption of sectional curvature bounded from below to the $\RCD(K,N)$ condition. Indeed from Kapovitch and Ketterer \cite{KK} we know that a $\CD(K,N)$ and CAT(k) space $(X,d,\mathcal{H}^N)$ is a space with bounded sectional curvature. More precisely, 
 
\begin{Thm}(\cite{KK}) Let $N\geq 2$ be a natural number and let $(X,d,\mathcal{H}^N)$ be a complete metric measure space which is $CAT(k)$ (has curvature bounded above by $k$ in the sense of Alexandrov) and satisfies $\CD(K,N)$. Then $k(N-1)\geq K$, and $(X,d)$ is an Alexandrov space of curvature
bounded below by $K-k(N-2)$. In particular, X is infinitesimally Hilbertian.
\label{cdcat}
\end{Thm}

Combining with Theorem \ref{Thm:c02} we have that solution of \eqref{equm} on such a space have strong unique continuation property. Note that the background measure here must be the proportional to the Hausdorff measure and that the situation is much more complicated otherwise.  

Indeed, from Kapovitch, Kell and Ketterer \cite{KKK}, if $(X,d,m)$ is a complete metric measure space which is $CAT(k)$ and satisfies $
\RCD(K,N)$, the current optimal regularity for the metric tensor is in $BV\cap C^0$. Thus the previous argument using Proposition \ref{gensuc} does not work. What is worse, from  Plis \cite{Pl} and Miller \cite{M}, even if the metric tensor is $C^{0,\alpha}$, there are counterexamples to weak unique continuation. 

If we only assume that the space has two sided Ricci curvature bound with a non-collapsed Ricci limit space structure, then let us recall a classical theorem in Cheeger-Colding theory(\cite{CC1,CC2,CC3}) considering the limit space of a non-collapsing sequence of manifolds with two sided bounds on Ricci curvature.
\begin{Thm}(\cite{CC1})
Let $(Y^n,y)$ be the Gromov-Hausdorff limit of a sequence $\{(M_i^n,p_i)\}$ for which there exists $v>0$ such that $Vol(B_1(p_i))\geq v>0$ for all $i$. If moreover we have that 
\begin{align*}
    |\Ric_{M^n_i}|\leq n-1,
\end{align*}
then we know that the regular points of Y form a $C^{1,\alpha}$-Riemannian manifold.
 \label{rL2sided}
\end{Thm}
 
 For a general higher dimensional manifold, we know that if the metric tensor is $C^{0,1}$, we have unique continuation theorem for harmonic functions. 

 Note that as the metric tensor around regular points in Y is $C^{1,\alpha}$, combining a similar argument in the proof Theorem \ref{rc01} and \cite{CN} we can show that,
 \begin{Thm}\label{Rictb}
If $(Y^n,y)$ is the Gromov-Hausdorff limit of a sequence, $\{(M_i^n,p_i)\}$ and there exists $v>0$ such that $Vol(B_1(p_i))\geq v>0$ and
\begin{align*}
    |\Ric_{M^n_i}|\leq n-1.
\end{align*}
Then we have strong unique continuation theorem for harmonic functions on $Y$ at regular points.
 \label{rlc}
\end{Thm}
 \begin{Rem}
  The Theorem \ref{Rictb} also implies the weak unique continuation on the whole space.
 \end{Rem}

  \section{Classification of collapsed $\RCD(K,2)$ spaces}
  In this section, we show that a collapsed $\RCD(K,2)$ space $(X,d,m)$ (i.e. where $m$ is not proportional to $\mathcal{H}^2$) is necessarily a point, a line, a ray, an interval or a circle. The results of this section were shown in \cite{KL}. Since the proof simplifies somewhat using the recently established non-branching property of $\RCD(K,N)$ spaces in \cite{D}, we include a self-contained proof. For the same result on Ricci limit space, cf. \cite{Ho}.
  
  We now recall some relevant definitions of $\RCD$ theory, beginning with the definition of tangent cones.
  
  \begin{Def}
  Let $(X,d,m)$ be an $\RCD(K,N)$ space and $\bar{x} \in X$. A pointed metric measure space $(Y, d_Y, m_Y, \bar{y})$ is called a tangent cone of $(X,d,m)$ at $\bar{x}$ if there exists a sequence of radii $r_{i} \to 0^+$ so that 
  \begin{equation*}
  (X, r_i^{-1}d, \frac{1}{m(B_{r_i}(\bar{x}))}m, \bar{x}) \to (Y, d_Y, m_Y, \bar{y})
  \end{equation*}
  in the pointed measured Gromov-Hausdorff topology. The collection of all tangent cones at $\bar{x}$ is denoted $Tan(X,d,m,\bar{x})$. 
  \end{Def}
  
  For the definition of pointed measured Gromov-Hausdorff convergence and relevant compactness results giving the existence of tangent cones at any point, see, for instance, \cite[Section 3]{GMS}. Note that in general, tangent cones may depend on the sequence $r_i \to 0^+$ and need not be unique at any given point. We now define regular points as follows.
  
  \begin{Def}
Let $(X,d,m)$ be an $\RCD(K,N)$ space. For any integer $k \geq 1$ and $x \in X$, we say that $x$ is a $k$-dimensional regular point if
\begin{equation*}
   Tan(X, d, m ,x) = \{(\mathbb{R}^k, d_{\mathbb{R}^k}, \frac{1}{\omega_k}\mathscr{H}^k, 0_k)\}.
\end{equation*}
We denote by $R_k$ the set of all $k$-dimensional regular points.
 \end{Def}
 
The following theorem about the structure of regular sets follows from the works of \cite{MN, BS}.
\begin{Thm}{\cite[Theorem 0.1]{BS}}
Let $(X,d,m)$ be an $\RCD(K,N)$ space. There exists a unique integer $n \in [1,N]$, called the essential dimension of $X$, so that 
\begin{equation*}
    m(X \backslash R_n) = 0.
\end{equation*}
\end{Thm}

By combining the results of \cite{BGHZ} and \cite{DG}, it is known that an $\RCD(K,2)$ space is collapsed iff its essential dimension is $1$ or $0$. The second case is trivial so we will only consider the case that the essential dimension is $1$. We will prove something slightly more general for any $\RCD(K,N)$ space with essential dimension $1$ which will imply the desired result. 
  
  \begin{Lem}\label{R1 interior}
    Let $(X,d,m)$ be an $\RCD(K,N)$ space and suppose $x \in R_1$, then $x$ lies in the interior of some geodesic.
  \end{Lem}
   
  \begin{proof}
    Suppose this is not the case, then there exists $r_i, \epsilon_i \to 0$ and $p_i, q_i \in X$ so that all of the following holds: 
    \begin{itemize}
        \item $d(p_i, x)$ = $d(q_i, x) = r_i$
        \item $d(p_i, x)+d(x,q_i) -d(p_i,q_i) = \epsilon_i r_i$
        \item $x$ does not lie on a geodesic $\overline{p_iq_i}$.
    \end{itemize}
    For each $i$, fix $z_i$ on $\overline{p_iq_i}$ so that $d(z_i, x) = d(\overline{p_iq_i}, x)$. There are two cases:
    
    \underline{Case 1:} $\liminf\limits_{i \to \infty} \frac{d(z_i, x)}{r_i} = 0$
    
    Up to choosing a subsequence, we may assume $s_i := d(z_i, x)$ is so that $\lim\limits_{i \to \infty} \frac{s_i}{r_i} = 0$ and $(X,\frac{1}{s_i}d, x)$ converges to $(\mathbb{R}^1, d_{E}, O)$ in the pointed Gromov-Hausdorff sense. Let $z_i \to z_\infty$. There is a line through $z_\infty$ from the limit of $\overline{p_iq_i}$. However, it is clearly that $O$ is not on this line since the distance between $O$ and the line must be $1$. This is a contradiction. 
    
    \underline{Case 2:} $\liminf\limits_{i \to \infty} \frac{d(z_i, x)}{r_i} > 0$
    
    Up to choosing a subsequence, we may assume $(X,\frac{1}{r_i}d, x)$ converges to $(\mathbb{R}^1, d_{E}, O)$ in the pointed Gromov-Hausdorff sense with $p_i \to p_\infty$, $q_i \to q_\infty$ and $\overline{p_iq_i}$ converging to some limit geodesic $\overline{p_\infty q_\infty}$. By our assumptions, we have
    \begin{itemize}
        \item $d(\overline{p_\infty q_\infty}, O)>0$
        \item $d(p_\infty,O)+d(O,q_\infty) = d(p_\infty,q_\infty)=2$
    \end{itemize}
    It is impossible to find three points on $\mathbb{R}$ which satisfy these properties. This gives a contradiction. 
  \end{proof}
  
  We now prove a classification result for $\RCD(K,N)$ spaces with essential dimension $1$.
  
  \begin{Thm}\label{1D classification}
    Let $(X,d,m)$ be an $\RCD(K,N)$ space with essential dimension $1$. $X$ is a line, ray, interval or circle. Moreover, $m = h\mathcal{H}^1 \ll \mathcal{H}^1$ where $h$ is a $\MCP(K,N)$ density.
  \end{Thm}
  
  \begin{proof}
    Fix any $x \in R_1$. By the previous lemma, we may assume that $x$ lies in the interior of some geodesic $\overline{pq}$. We claim that for any $y \in X$, any geodesic $\overline{yx}$ extends either $\overline{px}$ or $\overline{qx}$. Combining this with the non-branching property of $\RCD(K,N)$ spaces is enough to give the desired classification result of the theorem. 
    
    Fix any geodesic $\overline{yx}$. If $\overline{yx} \cap \overline{px} \neq \{x\}$, then $\overline{yx}$ necessarily extends $\overline{px}$ by the non-branching property of $X$ and we are done. The same is true if we replace $\overline{px}$ with $\overline{qx}$ in the previous sentence. Therefore, we may assume $\overline{yx} \cap \overline{px} = \{x\} = \overline{yx} \cap \overline{qx}$.
    
    Fix $r_i \to 0$ so that $(X,\frac{1}{r_i}d,x)$ converges to $(\mathbb{R}, d_{E}, O)$ in the pointed Gromov-Hausdorff sense. For each $i$ choose $y_i$ on $\overline{yx}$ so that $d(y_i,x)=r_i$. Similarly, choose $p_i$ and $q_i$ on $\overline{px}$ and $\overline{qx}$ respective so that $d(p_i,x)=d(q_i,x)=r_i$. Since the blow-up converges to $\mathbb{R}$, we have either
    \begin{equation*}
        \lim\limits_{i \to \infty} \frac{d(y_i,p_i)}{r_i}=0 \text{ and } \lim\limits_{i \to \infty} \frac{d(y_i,q_i)}{r_i}=2 
    \end{equation*}
    or
     \begin{equation*}
        \lim\limits_{i \to \infty} \frac{d(y_i,p_i)}{r_i}=2 \text{ and } \lim\limits_{i \to \infty} \frac{d(y_i,q_i)}{r_i}=0.
    \end{equation*}
    Without loss of generality we assume the first case. We show $\overline{yx}$ extends $\overline{px}$. Suppose not, then by non-branching any geodesic $\overline{y_iq_i}$ can only intersect $\overline{p_iq_i}$ at $q_i$. In particular, $x$ is not in the interior of $\overline{y_iq_i}$. Repeating the same argument as in the previous lemma gives a contradiction. 
    
    The claim about $m$ follows from, for example, \cite{CaM2}.
  \end{proof}
  
  Combined with the discussion in the beginning of the section immediately gives the following corollary.
  \begin{Cor}
    Any collapsed $\RCD(K,2)$ space which is not a point is isometric to a line, a ray, an interval or a circle. 
    \label{clps}
  \end{Cor}

  \section{Proof of Theorem \ref{thm:c01}}
Let's first we recall the corresponding unique continuation theorem in $\R^2$ for elliptic equations of divergence form with merely measurable coefficients.
 \begin{Thm}[\cite{AM, Al, Sc}]
 Given a solution u of equation 
\begin{align}
    Lu = -D_\alpha(a^{\alpha\beta}(z)D_\beta u+b^\alpha(z)u)+c^\alpha(z)D_\alpha u+e(z)u=0
\end{align}
 on a simply connected domain $\Omega$ in the plane where $A=(a^{\alpha\beta})$ is a $2\times 2$ symmetric matrix of $L^\infty(\Omega)$ coefficients satisfying for some constant $0<\lambda\leq 1$ the uniform elliptic condition 
 \begin{align}
     \lambda |\xi|^2\leq \sum_{\alpha,\beta=1}^2a^{\alpha\beta}(z)\xi_\alpha\xi_\beta\leq   \lambda^{-1} |\xi|^2, \text{ for all } z\in \Omega,\,\,\xi\in \R^2.
 \end{align}
If there exists a point $x\in \Omega$ such that
 \begin{align}
     u(z)=o(|d(x,z)|^n) \text{ for all }n\in\N .
 \end{align}
 Then we have that $u\equiv 0$.\label{thm:uc}
  \end{Thm}
 
  The key of the proof is to reduce the problem to an elliptic system which is handled by Bers and Nirenberg \cite{BN}. For more details, see, for example, Schulz \cite{Sc}.
  
  We now prove Theorem \ref{thm:c01} using the results from the previous sections. 
   \begin{Thm}
  Given a harmonic function $\phi\in W^{1,2}(\Omega)$ on a connected open subset $\Omega$ of a $\RCD(K,2)$ space $X$, if we have that $\phi|_{B_r(x)}=0$ for some $x\in \Omega$ and $r>0$, then $\phi\equiv 0$.
  \label{rc01}
 \end{Thm}
  \begin{proof}
    We first consider the non-collapsed case. Using Proposition \ref{pro:raequ}, we may assume we are working with an Alexandrov surface. By Proposition \ref{pro:c1}, we know that there exists coordinate chart around any regular point so that the metric tensor is $C^{\frac{1}{2}}$ outside a measure zero set, that is, there exists some $S \subset U$ with $\mathcal{H}^2(S)=0$ so that
  \begin{align}
      |g_{ij}(x)-g_{ij}(y)|=O(|xy|^{\frac{1}{2}}),\,\,for \,\,x\in U\setminus S,\,\,y\in U.
  \end{align}
  
 Therefore, given any solution of \eqref{equm} on an Alexandrov surface, the function in the local chart around a regular point satisfies the assumptions of Theorem \ref{thm:uc}. 
 
 Given any function $\phi\in W^{1,2}(\Omega)$ which is a solution of \eqref{equm} with $\phi|_{B_r(x)}=0$ for some $x\in \Omega$ and $r>0$. We show that $\phi\equiv 0$. It suffices to show that $\mathcal{H}^2(\Omega\setminus\phi^{-1}(0))=0$.
 
Assume otherwise, then $\partial \phi^{-1}(0)$ is nonempty in $\Omega$. Given any point $x_0\in \overline{\phi^{-1}(0)}\setminus \phi^{-1}(0)\subset\partial \phi^{-1}(0)$, if $x_0$ is regular, then there exists a local charts around $x_0$. Given a sequence of points $x_i\in Int(\phi^{-1}(0))$ converging to $x_0$. Then for large enough i, $x_i$ is in the chart. So we can use Theorem \ref{thm:uc} to conclude that $x_0$ is actually in $Int(\phi^{-1}(0))$. This argument shows that any point in $ \overline{\phi^{-1}(0)}\setminus \phi^{-1}(0)\subset\partial \phi^{-1}(0)$ must be singular. 

Now consider any geodesic from a regular point in $B_r(x)\cap\phi^{-1}(0)$ to a regular point in $\Omega\setminus\phi^{-1}(0)$. Every point on the geodesic must be regular by \cite{P1}. However, the geodesic cross $ \overline{\phi^{-1}(0)}\setminus \phi^{-1}(0)\subset\partial \phi^{-1}(0)$ which is a contradiction in view of the previous paragraph. 

If $(X,d,m)$ is a collapsed $\RCD(K,2)$ space, that is $m$ is not proportional to $\mathcal{H}^2$, then from the Corollary \ref{clps} we know that the space is isometric to a point, a line, a ray, an interval or a circle. Locally around each point away from the end points of such space (other than the point, which is trivial) is isometric to $(a,b)$ around an interior point. Denote $I=(a,b)$ and $\tilde{m}|_{I}=h\mathcal{L}^1|_{I}$. Then since $\RCD(K,2)$ implies $\CD_{loc}(K,2)$, we know that $(I,|\cdot|,h\mathcal{L}^1|_{I})$ is a $\CD(K,2)$ space. This shows that (cf. \cite{CaM1,CaM2}) \begin{align}
    h((1-t)x_0+tx_1)\geq \sigma_{k,1}^{(1-t)}(|x_1-x_0|)h(x_0)+\sigma_{k,1}^{(t)}(|x_1-x_0|)h(x_1)
\end{align}where 
\begin{equation}
 \sigma_{k,1}^{(t)}(\theta)=
 \left\{\begin{array}{ccc}
 \infty,&if\,\,k\theta^2\geq \pi^2\\ \frac{\sin(t\theta\sqrt{k})}{\sin(\theta\sqrt{k})},&if\,\,0<k\theta^2<\pi^2\\
    t,&if\,\,k\theta^2<0\,\,and\,\,N=1,\,\,or\,\,k\theta^2=0\\
     \frac{\sinh(t\theta\sqrt{k})}{\sinh(\theta\sqrt{k})},&if\,\,k\theta^2\leq 0\,\,and\,\,N>1
    \end{array}
    \right.
\end{equation}
This shows that 
\begin{align}
    \frac{h((1-t)x_0+tx_1)-h(x_0)}{t|x_0-x_1|}\geq \frac{\sigma_{k,1}^{(1-t)}(|x_1-x_0|)-1}{t|x_0-x_1|}h(x_0)+\frac{\sigma_{k,1}^{(t)}(|x_1-x_0|)}{{t|x_0-x_1|}}h(x_1),\\ 
    \frac{h(x_1)-h((1-t)x_0+tx_1)}{(1-t)|x_0-x_1|}\leq -\frac{\sigma_{k,1}^{(1-t)}(|x_1-x_0|)}{(1-t)|x_0-x_1|}h(x_0)+\frac{1-\sigma_{k,1}^{(t)}(|x_1-x_0|)}{(1-t)|x_0-x_1|}h(x_1),
\end{align}
which means that h is $C^{0,1}$ and positive away from end points, see also \cite{KKK}. So given any function f on  $(I,|\cdot|,h\mathcal{L}^1|_{I})$ which is a solution of \eqref{equm}, we can extend it to a function $\tilde{f}$ on $ (I\times I,|\cdot|,h(x_1)\mathcal{L}^2|_{I\times I})$ which is also a solution of \eqref{equm} with $\tilde{f}(x_1,x_2)=f(x_1)$, $\tilde{a}(x_1,x_2)=a(x_1)$,  $\tilde{b}^1(x_1,x_2)=b(x_1)$, $\tilde{b}^2(x_1,x_2)=0$, $\tilde{c}^1(x_1,x_2)=c(x_1)$, $\tilde{c}^2(x_1,x_2)=0$, $\tilde{e}(x_1,x_2)=e(x_1)$. So $\tilde{f}$ still satisfies the condition of Theorem \ref{thm:uc}. This shows that the unique continuation also holds for collapsed $\RCD(K,2)$ spaces.

  \end{proof}

 \section{Counterexample For strong unique continuation}
 
 In this section we show that the strong unique continuation property for harmonic functions fails on some $\RCD(K,4)$ space for all $K \in \R$. In the first two steps of the proof, we will construct the spaces on which we have our counterexamples. The examples considered dates back to the discussion on metric horn given by Cheeger and Colding in \cite{CC1}. In the last two steps, we will check the failure of unique continuation on such spaces. The techniques used are inspired by \cite{Ka1,Ka2,CM1,WZ}. 
 \begin{Thm}
For all $K\in\R$, there exists an $\RCD(K,4)$ space X, an open subset $\Omega\subset X$ and a harmonic function $u:\Omega\to \R$ which does not have the strong unique continuation property, that is, u vanishes to infinite order at some $x_0\in \Omega$ but $u\not\equiv 0$. Moreover, if $K\leq 0$, $u$ may be chosen to be a global harmonic function.
 \end{Thm}\begin{proof}

For each $n \geq 3$, define the metric $g_{\ep,n}$ as 
\begin{align}
    g_{\ep,n}=dr^2+(\frac{1}{2}r^{1+\ep})^2g_{S^{n-1}}.
    \end{align}
 
For a warped product of the form $g=dr^2+\phi(r)^2g^{S^{n-1}}$, we have that 
 \begin{align}
     \Ric_g=-(n-1)\frac{\phi''}{\phi}dr\otimes dr+((n-2)(1-(\phi')^2)-\phi''\phi)g_{S^{n-1}}.
 \end{align}
Furthermore, for a weighted Riemannian manifold $(M,dr^2+\phi(r)^2g^{S^{n-1}},e^{-\chi}dvol_g)$ and $N\in \N\setminus \{n\}$, the  $N$-Ricci curvature (cf. \cite{Li,BaE,Ba,Qi}) is defined as
\begin{align}
    \Ric_N:=-(n-1)\frac{\phi''}{\phi}dr\otimes dr+((n-2)(1-(\phi')^2)-\phi''\phi)g_{S^{n-1}}+Hess\,\chi-\frac{d\chi\otimes d\chi}{N-n}.
\end{align}
 
Then we have that for $(\R^n\setminus\{0\},\,\,g_{\ep,n},\,\,e^{(N-n)(1-\eta)log(r)}dvol),$\begin{align}
\begin{split}
    \Ric_N=&\frac{(N-n)\eta(1-\eta)-(n-1)\ep(1+\ep)}{r^2}dr\otimes dr\\ &+(n-2-\frac{(1+\ep)}{4}(n-2+(n-1)\ep+(N-n)(1-\eta))r^{2\ep})g_{S^{n-1}}.
    \end{split}
\end{align}
From now on we assume $N=n+1=4$, then
\begin{align}
    &Ric_4=\frac{\eta(1-\eta)-2\ep(1+\ep)}{r^2}dr\otimes dr+(1-\frac{(1+\ep)(2+2\ep-\eta)}{4} r^{2\ep})g_{S^{2}}.
\end{align}
By carefully choosing $\ep$ and $\eta$, we have that $Ric_4\geq Kg_{\ep,n}$ on $\{0<r\leq C(K)\}$ for some $C(K)>0$. Since the constructed space should be complete, we need to deal with the vertex and the set $\{r\geq C(K)\}$. In step 1, we will modify the metric in $\{r\geq C(K)\}$ to get a $\RCD(K,4)$ space which has the same metric on $\{0< r\leq C(K)\}$. We will denote the space as $X$.\\

 \textit{Step 1. Gluing for large $r$.}
 
First we show that if $K>0$, we can add a cap to close the end so that the resulting space satisfies $\Ric_4\geq Kg$. The technique is similar to that in the last section of \cite{WZ}. 
 
We shall construct the metric by gluing a smooth small sphere metric at the end. The rotational symmetry will be preserved in the process.  Thus we are actually gluing the two warping functions
\begin{align*}
    \phi_{g_{\ep,n}}=\frac{1}{2}r^{1+\ep} \; \;   \quad \textrm{and} \quad \phi_{a^{-2} g_{round}}=\frac{\sin\left(ar\right)}{a}, \; a > 1. 
\end{align*}
It is clear that $\Ric_4\geq Kg$ in the later case given $a\gg 1$.

Choose a small positive constant $\rho$ and define
\begin{align}
    a =\frac{\sqrt{4-(1+\ep)^2 \rho^{2\ep}}}{\rho^{1+\ep}}, \quad \xi= \frac{\arccos (\frac{1+\ep}{2}\rho^{\ep})}{a}. 
\end{align}

Then we have
\begin{align}
 \frac{1}{2}\rho^{1+\ep} =\frac{\sin\left(a \xi \right)}{a}, \quad \frac{1+\ep}{2}\rho^{\ep}=\cos\left(a \xi \right).\label{eqn:OI20_4z}
\end{align}
It is clear that $a \to +\infty$ and $\xi \to 0^{+}$, as $\rho \to 0^{+}$.

Take a smooth function 
\begin{align*}
\psi:\R\to \R,\,\,\psi(x)=\left\{\begin{array}{ccc}0, \,\,\,\,\quad\forall x\leq 0,\\1, \,\,\,\,\quad\forall x\geq 1.\end{array}\right.
\end{align*}
Denote
 \begin{align}
 l(t)= \left(1-\psi\left(\frac{t-\rho}{\zeta}\right)\right)\frac{\ep(1+\ep)}{2}t^{\ep-1}
-\psi\left(\frac{t-\rho}{\zeta}\right)a\,\sin\left(a\,(t-\rho-\zeta+\xi)\right). 
\label{eqn:OI20_5z}
\end{align} 
From its definition, for each integer $k \geq 0$, we have
\begin{align}
&l^{(k)}(\rho)=\frac{\Pi_{i=0}^{k+1}(1+\ep-i)}{2}\rho^{\ep-k-1} \label{eqn:OI26_1}\\
&l^{(2k)}(\rho+\zeta)=(-1)^{k+1} a^{2k+1} \sin (a \xi),  \quad l^{(2k+1)}(\rho+\zeta)=(-1)^{k+1} a^{2k+2} \cos (a \xi).   \label{eqn:OI26_2}
\end{align}
Then for any small constant $\zeta,\kappa\leq \frac{1}{100}$ we can construct the function $\phi=\phi_{\xi,\zeta,\kappa}$ as follows.  
\begin{itemize}
\item[(1).]  If $r \in [0,\rho]$, then 
       \begin{align}
        \phi = \frac{1}{2}r^{1+\ep}. 
        \label{eqn:OI27_3}   
      \end{align}
\item[(2).]  If $r \in [\rho,\rho+\zeta]$, then 
     \begin{align}
        \phi = \frac{1}{2}\rho^{1+\ep}+(r-\rho) \frac{1+\ep}{2}\rho^{\ep}+\int_{s=\rho}^{r}\int_{t=\rho}^{s}l(t). 
     \label{eqn:OI27_1}   
     \end{align}
\item[(3).]  If $r \in [\rho+\zeta,\rho+\zeta+\kappa]$, then 
        \begin{align}
          \phi&=\frac{\sin\left(a \xi \right)}{a}+(r-\rho-\zeta)\cos\left(a \xi \right)+\left(1-\psi\left(\frac{r-(\rho+\zeta)}{\kappa}\right)\right)\int_{s=\rho}^{r}\int_{t=\rho}^{s}l(t) \notag\\
                &\quad+\psi\left(\frac{r-(\rho+\zeta)}{\kappa}\right)\int_{s=\rho+\zeta}^{r}\int_{t=\rho+\zeta}^{s}l(t)+\left(1-\psi\left(\frac{r-(\rho+\zeta)}{\kappa}\right)\right)\zeta\,\cos\left(a \xi \right). 
                \label{eqn:OI27_2}
        \end{align}
\item[(4).]  If $r \in [\rho+\zeta+\kappa,\rho+\zeta-\xi+\frac{\pi}{a}]$,  then 
       \begin{align}
        \phi=\frac{\sin\left(a\,(r-\rho-\zeta+\xi)\right)}{a}. 
        \label{eqn:OI27_4}   
      \end{align}
\end{itemize}
In light of \eqref{eqn:OI20_4z}, \eqref{eqn:OI26_1} and \eqref{eqn:OI26_2}, direct calculation implies that $\phi$ is smooth at $r=\rho, \rho+\zeta$ and $\rho+\zeta+\kappa$. 
Consequently, $\phi$ is smooth on $[0, \rho+\zeta-\xi+\frac{\pi}{a}]$.

Then we also glue the measure on the spherical part with $m=C dvol_g$. As \begin{align}
    Hess\,\chi-\frac{d\chi\otimes d\chi}{N-n}=(\chi_{rr}-\chi_r^2)dr\otimes dr+\phi\phi'\chi_rg_{S^2}.
    \end{align}

And we start from $\chi_0(r)=-(1-\eta)log(r)$. By a similar trick as above we can set $\chi_{rr}$ to oscillate on $[\rho+\zeta+\kappa,\rho+\zeta+3\kappa]$ such that $\chi_{rr}(\rho+\zeta+\kappa)=(1-\eta)\frac{1}{(\rho+\zeta+\kappa)^2}$,\,$\chi^{(n)}(\rho+\zeta+3\kappa(=0$ for $n\geq 1$ and $\chi_{rr}(s)\geq 0,\,\chi_r(s)\leq 0$ on $s\in[\rho+\zeta+\kappa,\rho+\zeta+3\kappa]$. 

Now we cliam that we can choose $\chi_{rr}$ properly such that $\chi_{rr}-\chi_r^2\geq 0$ on $[\rho+\zeta+\kappa,\rho+\zeta+3\kappa]$. 

Define\begin{align}
    \chi_{rr}(r)=\psi_{\rho+\zeta+\kappa,\rho+\zeta+\kappa+\ep}(r)\chi_{0,rr}(r)+(1-\psi_{\rho+\zeta+\kappa,\rho+\zeta+\kappa+\ep})\psi_{\rho+\zeta+3\kappa-\ep,\rho+\zeta+3\kappa}(r)K, 
\end{align}
where \begin{align}
    \varphi_{a,l}(x)=\left\{\begin{array}{ccc}0&,\,\,x\leq a\\
    e^{-(x-a)^{-2}}&,\,\,x>a.
    \end{array}\right.
    \\   \varphi_{a,r}(x)=\left\{\begin{array}{ccc}e^{-(x-a)^{-2}}&,\,\,x\leq a\\
    0&,\,\,x>a.
    \end{array}\right.
    \\\psi_{a,b}(x)=\frac{\varphi_{b,r}(x)}{\varphi_{a,l}(x)+\varphi_{b,r}(x)}.
\end{align} Here we choose $\ep\leq\frac{\kappa}{100} $ and $K$ such that  $\chi_r(\rho+\zeta+3\kappa)=0$.

As we know that $(\chi_{rr}-\chi_r^2)(\rho+\zeta+\kappa)>0$, we can get that $\chi_{rr}-\chi_r^2\geq 0$ on $[\rho+\zeta+\kappa,\rho+\zeta+\kappa+\ep]$ once we choose $\ep$ small enough. Also on $[\rho+\zeta+\kappa+\ep,\rho+\zeta+3\kappa-\ep]$ we also have $\chi_{rr}-\chi_r^2\geq 0$ as $\chi_{rr}\equiv K$ on $[\rho+\zeta+\kappa+\ep,\rho+\zeta+3\kappa-\ep]$ and $(\chi'_{r})^2$ is decreasing. 

On $[\rho+\zeta+3\kappa-\ep,\rho+\zeta+3\kappa]$ as  $-\chi_r(s)=\int_{s}^{\rho+\zeta+3\kappa}\chi_{rr}(t)dt$ and $K\leq \frac{3}{(\rho+\xi+\kappa)\kappa}$ we have that $\chi_{rr}(s)-\chi_r^2(s)\geq \chi_{rr}(s)(1-\ep^2\chi_{rr}(s))\geq \chi_{rr}(s)(1-\ep^2 K)\geq 0$ as $\ep\leq\frac{\kappa}{100}$.
Thus we conclude the such $\chi$ satisfies that $\chi_{rr}-\chi_r^2\geq 0$. And by choosing $\kappa$ small enough and $\rho<<1$ we can get that $\phi\phi'\chi_r\leq \frac{(\frac{\sin\left(a\,(r-\rho-\zeta+\xi)\right)}{a})^2}{4}$.

\textbf{Claim:\;}
\textit{For each fixed small  $\rho\in (0, 1)$and sufficiently small  $\kappa$ and $\zeta \ll \kappa^2$, the smooth function $\phi=\phi_{\xi, \zeta, \kappa}$ constructed above satisfies }
\begin{align}
    \Ric_4&=-2\frac{\phi''}{\phi}dr\otimes dr+((1-(\phi')^2)-\phi''\phi)g_{S^{2}}+Hess\,\chi-{d\chi\otimes d\chi}\notag\\&\geq Kg=K(dr^2+\phi(r)^2g_{S^{2}}) \quad on\,\,r\leq \rho+\zeta-\xi+\frac{\pi}{a}.
    \label{eqn:OI26_5}
\end{align}

By discussion at the beginning of this section, it is clear that \eqref{eqn:OI26_5} is satisfied in cases (1) and (4).  
Thus we only focus on the proof of \eqref{eqn:OI26_5} in cases (2) and (3). 

We first show \eqref{eqn:OI26_5}. 
On the interval $ [\rho,\rho+\zeta]$, it follows from \eqref{eqn:OI20_5z} and \eqref{eqn:OI27_1} that
\begin{align*}
  -\phi''&=-l(r)=-\left(1-\psi\right)\frac{\ep(1+\ep)}{2}\rho^{\ep-1}
+\psi a\,\sin\left(a\,\xi\right)+ O(\zeta), \\
\phi'&=\frac{1+\ep}{2}\rho^{\ep} + O(\zeta), \\
  \phi&=\frac{1}{2}\rho^{1+\ep} + O(\zeta). 
\end{align*}
Applying \eqref{eqn:OI20_4z} on the above equations, we have, given $\rho,\ep$ sufficiently small,
\begin{align*}
  -\frac{\phi''}{\phi}\geq -(1+O(\zeta))\frac{\ep(1+\ep)}{\rho^2},\\
  (1-(\phi')^2)-\phi''\phi\geq \frac{1}{4}.
\end{align*}
Thus the $\psi$ term would dominate as before and we obtain \eqref{eqn:OI26_5} in case (2).
On the interval $[\rho+\zeta,\rho+\zeta+\kappa]$,  the second derivative $\phi''$ can be expressed as
\begin{align*} 
&\quad l(r)+\bigg[\psi\left(\frac{r-(\rho+\zeta)}{\kappa}\right)\left(\int_{s=\rho+\zeta}^{r}\int_{t=\rho+\zeta}^{s}l(t)-\int_{s=\rho}^{r}\int_{t=\rho}^{s}l(t)\right)\\
&\quad \quad +\left(1-\psi\left(\frac{r-(\rho+\zeta)}{\kappa}\right)\right)\zeta\,\cos\left(a \xi \right) \bigg]''\\
&=l(r)+\frac{1}{\kappa^2}\psi''\left(\frac{r-(\rho+\zeta)}{\kappa}\right)\left(\int_{s=\rho+\zeta}^{r}\int_{t=\rho+\zeta}^{s}l(t)-\int_{s=\rho}^{r}\int_{t=\rho}^{s}l(t)-\zeta\,\cos\left(a \xi \right) \right)\notag\\
&\quad \quad +2\frac{\psi'\left(\frac{r-(\rho+\zeta)}{\kappa}\right)}{\kappa}\left(\int_{t=\rho+\zeta}^{r}l(t)-\int_{t=\rho}^{r}l(t)\right)\\
&=l(r) + O\left(\frac{\zeta}{\kappa^2} \right) + O\left(\frac{\zeta}{\kappa} \right). 
\end{align*}
Since $\zeta \ll \kappa^2$ according to our choice, it follows from \eqref{eqn:OI20_5z},  \eqref{eqn:OI27_2} and the above inequality that 
\begin{align*}
-\phi''&=a \sin (a\xi) + O(\kappa)+O\left(\frac{\zeta}{\kappa} \right), \\
\phi'&=\cos (a\xi) + O(\kappa),\\
\phi&=\frac{\sin (a\xi)}{a} + O(\kappa). 
\end{align*}
Recalling that both $\kappa$ and $\frac{\zeta}{\kappa}$ are very small,  we immediately obtain \eqref{eqn:OI26_5} holds on $[\rho+\zeta, \rho+\zeta+\kappa]$, i.e., in case (3).  The proof of the claim is complete. \\

For $K\leq 0$, we will construct an open space with $\Ric_4\geq 0$ and a unique tangent cone at infinity. This is done by gluing the warping functions
\begin{align*}
    \phi_{g_{\ep,n}}=\frac{1}{2}r^{1+\ep} \; \;   \quad \textrm{and} \quad \phi_{a}=ar. 
\end{align*}

Choose a small positive constant $\rho$ and define
\begin{align}
    a =\frac{1+\ep}{2}\rho^{\ep}, \quad \xi= \frac{\rho}{1+\ep}. 
\end{align}

Then we have
\begin{align}
 \frac{1}{2}\rho^{1+\ep} =a\xi, \quad \frac{1+\ep}{2}\rho^{\ep}=a.\label{eqn:OI20_4za}
\end{align}
It is clear that $a \to 0^{+}$ and $\xi \to 0^{+}$, as $\rho \to 0^{+}$.

Take a smooth function 
\begin{align*}
\psi:\R\to \R,\,\,\psi(x)=\left\{\begin{array}{ccc}0, \,\,\,\,\quad\forall x\leq 0 \\1, \,\,\,\,\quad\forall x\geq 1.\end{array}\right.
\end{align*}
Denote
 \begin{align}
 l(t)= \left(1-\psi\left(\frac{t-\rho}{\zeta}\right)\right)\frac{\ep(1+\ep)}{2}t^{\ep-1}. 
\label{eqn:OI20_5z1}
\end{align} 
Similar as before, we define \begin{itemize}
\item[(1)]  If $r \in [0,\rho]$, then 
       \begin{align}
        \phi = \frac{1}{2}r^{1+\ep}. 
        \label{eqn:OI27_31}   
      \end{align}
\item[(2)]  If $r \in [\rho,\rho+\zeta]$, then 
     \begin{align}
        \phi = \frac{1}{2}\rho^{1+\ep}+(r-\rho) \frac{1+\ep}{2}\rho^{\ep}+\int_{s=\rho}^{r}\int_{t=\rho}^{s}l(t). 
     \label{eqn:OI27_11}   
     \end{align}
\item[(3)]  If $r \in [\rho+\zeta,\rho+\zeta+\kappa]$, then 
        \begin{align}
          \phi&=a\xi+a(r-\rho-\zeta)+\left(1-\psi\left(\frac{r-(\rho+\zeta)}{\kappa}\right)\right)\int_{s=\rho}^{r}\int_{t=\rho}^{s}l(t)+\left(1-\psi\left(\frac{r-(\rho+\zeta)}{\kappa}\right)\right)a\zeta. 
                \label{eqn:OI27_21}
        \end{align}
\item[(4)]  If $r \in [\rho+\zeta+\kappa,+\infty)$,  then 
       \begin{align}
        \phi=a\,(r-\rho-\zeta+\xi). 
        \label{eqn:OI27_41}   
      \end{align}
\end{itemize}
And similar as before we also glue the measure with $Cdvol_g$. Similar calculation shows that $\Ric_4\geq 0$ in this case.\\

 \textit{Step 2. Verifying the $\RCD$ condition.}
 
 Next we show that the space is an $\RCD(K,N)$ space. From \cite{CaMi} we know that it suffices to show that the space is $\CD_{loc}(K,N)$. Recall that \begin{Def}
Given $K \in \R$, $N \in (1,\infty]$ and $\mathcal{N} \in (0,\infty]$, define:
\[
D_{K,\mathcal{N}} := \begin{cases}  \frac{\pi}{\sqrt{K/\mathcal{N}}}  & K > 0 \;,\; \mathcal{N} < \infty \\ +\infty & \text{otherwise.} \end{cases} .
\]
In addition, given $t \in [0,1]$ and $0 < \theta < D_{K,\mathcal{N}}$, define:
\[
\sigma^{(t)}_{K,\mathcal{N}}(\theta) := \frac{\sin(t \theta \sqrt{\frac{K}{\mathcal{N}}})}{\sin(\theta \sqrt{\frac{K}{\mathcal{N}}})} = 
\begin{cases}   
\frac{\sin(t \theta \sqrt{\frac{K}{\mathcal{N}}})}{\sin(\theta \sqrt{\frac{K}{\mathcal{N}}})}  & K > 0 \;,\; \mathcal{N} < \infty, \\
t & K = 0 \text{ or }\mathcal{N} = \infty \\
 \frac{\sinh(t \theta \sqrt{\frac{-K}{\mathcal{N}}})}{\sinh(\theta \sqrt{\frac{-K}{\mathcal{N}}})} & K < 0 \;,\;\mathcal{N} < \infty 
\end{cases} ,
\]
and set $\sigma^{(t)}_{K,\mathcal{N}}(0) = t$ and $\sigma^{(t)}_{K,\mathcal{N}}(\theta) = +\infty$ for $\theta \geq D_{K,\mathcal{N}}$. \\
And define
\[
\tau_{K,N}^{(t)}(\theta) := t^{\frac{1}{N}} \sigma_{K,N-1}^{(t)}(\theta)^{1 - \frac{1}{N}} .
\]
When $N=1$, set $\tau^{(t)}_{K,1}(\theta) = t$ if $K \leq 0$ and $\tau^{(t)}_{K,1}(\theta) = +\infty$ if $K > 0$. 
\end{Def}

 \begin{Def}(\cite{Stu1,Stu2,LV})
A metric measure space $(X,d,m)$ is said to satisfy $\CD_{loc}(K,N)$ if for any $o \in \supp(m)$, there exists a neighborhood $X_o \subset X$ of $o$, so that for all $\mu_0,\mu_1 \in P_2(X,d,m)$ supported in $X_o$, there exists $\nu \in Opt(\mu_0,\mu_1)$ so that for all $t\in[0,1]$, $\mu_t := (e_t)_{\#} \nu \ll m$, and for all $N' \geq N$,
\begin{equation} \label{CDLOC}
\mathcal{E}_{N'}(\mu_t) \geq \int_{X \times X} (\tau^{(1-t)}_{K,N'}(d(x_0,x_1)) \rho_0^{-1/N'}(x_0) + \tau^{(t)}_{K,N'}(d(x_0,x_1)) \rho_1^{-1/N'}(x_1)) \pi(dx_0,dx_1) ,
\end{equation}
where $\pi = (e_0,e_1)_{\sharp}(\nu)$ and $\mu_i = \rho_i m$ for $i=0,1$.

\[
\mathcal{E}_N(\mu) := \int \rho^{1 - \frac{1}{N}} dm.
\]

\end{Def}
 Since the space has a weighted manifold structure with the corresponding $4$-Ricci curvature lower bound outside the tip, it is clear that for any $x \in X\setminus \{V\}$, the $\CD_{loc}$ condition is satisfied. 
 
 Thus it suffices to check the $\CD_{loc}$ condition around the vertex $V$.
 
 \begin{Lem}
On $X=(\R^3,g_{\ep,3})$ with \begin{align}
    g_{\ep,3}=dr^2+(\frac{1}{2}r^{1+\ep})^2g_{S^{2}},
    \end{align}  the vertex $V$ cannot be in the interior of any geodesic between any $x,y\in B_\eta(V)$, where $\eta>0$.\label{geo}
 \end{Lem}
 
 \begin{proof}
     It suffices to show that for small enough $\eta>0$, any $x,y\in B_\eta(V)$, any curve $\gamma:[0,1]\to X$ with $\gamma(0)=x,\,\gamma(1)=y$ and $\gamma(s)=V$ for some $s\in [0,1]$, the length $|\gamma|>d(x,y)$. 
     
     Since we know that $|\gamma|\geq d(V,x)+d(V,y)=r(x)+r(y)$ and as \begin{align}
     d(x,y)=d((r(x),\theta(x)),(r(y),\theta(y)))\leq |r(x)-r(y)|+\frac{\pi}{2}\min\{r(x),r(y)\}^{1+\ep}\leq r.     \end{align}
     Given $\eta$ small enough, for $x,y\in B_\eta(V)$ we have that 
     \begin{equation*}
     d(x,y) \leq |r(x)-r(y)|+\frac{\pi}{2}\min(\{r(x),r(y)\})^{1+\ep}< r(x)+r(y)\leq |\gamma|.
     \end{equation*} This shows that $\gamma$ cannot be a geodesic between $x$ and $y$.
     
 \end{proof}
 
Now we show the $\CD_{loc}$ condition around the vertex $V$. Given any $\mu_0,\mu_1 \in P_2(X,d,m)$ supported in $B_\eta(V)$, where $\eta$ is given in \eqref{geo}. From \cite{Lis} we know that there exists a $\pi\in P(Geo(X))$ such that $(e_t)_{\#}(\pi)=\mu_t$. As we know that from Lemma \ref{geo} all geodesics do not have vertex $V$ as an interior point and also essentially do not have $V$ as the boundary point as both $\mu_1,\mu_2$ are absolutely continuous with respect to $m=r^{(1-\eta)}dvol_{g_0}$. Thus as $X\setminus \{V\}$ has $4$-Ricci curvature lower bound, \eqref{CDLOC} is satisfied in this case.

By \cite[Proposition 6.7]{KK}, it is known that a $\CD(K,N)$ space $(X,d,m)$ is infinitesimally Hilbertian iff $m$-a.e. point $x \in X$ has a Euclidean space as one of its tangent cones. This clearly applies for the constructed example since it is a weighted Riemannian manifold away from $V$ and so the $\RCD(K,4)$ condition is verified.\\

 \textit{Step 3. Existence of non-trivial harmonic functions.} 
 
 Next we show that there exists at least one non-trivial harmonic function locally for the spaces constructed above. In the case of $K\leq 0$, we show the existence of a global harmonic function.
  
  First we show that we can construct harmonic function in $B_s(0)$ with any given continuous boundary value using Perron's method. For more discussions on Perron's method on metric spaces see for example \cite{BB}. 
 
 Given any continuous function $g:\partial B_s(0)\to \R$, consider \begin{align}
     S_g=\{h:B_s(0)\to \R\,|\, \Delta_X h\geq 0,\,\,h|_{\partial B_s(0)}\leq g\}.
 \end{align}
 Define
 \begin{align}
     f(x)=\sup_{v\in S_g}v(x).
 \end{align}
 From Perron's method and as all points in $\partial B_s(0)$ are regular, we have that f is continuous in $B_s(0)$, harmonic in  $0<r<s$ and $f|_{\partial B_s(0)}=g$. We show that it is harmonic in $B_s(0)$. Given any compactly supported function $u:B_s(0)\to \R$, we have, from Cheng-Yau estimate, 
 \begin{align}
     \sup_{\partial B_r(0)}|\na f|\leq \frac{C\sup_{B_{2r}(0)}(|f-f(0)|)}{r}\leq \frac{Ce^{-\frac{\ep}{4} (log\,r)^2}}{r},
 \end{align}
 where the last inequality comes from \eqref{qest}. So we have that 
 \begin{align}
    \int_{B_s(0)}\la \na u,\,\na f\ra\,d\nu= \lim_{t\to 0} \int_{B_s(0)\setminus B_t(0)}\la \na u,\,\na f\ra\,d\nu=\lim_{t\to 0} \int_{\partial B_t(0)}u\la\na f,\,\partial_r\ra d\nu=0.
 \end{align}
 This shows that $f$ is harmonic in $B_s(0)$ with $f|_{\partial B_s(0)}=g$. In particular, if we choose $g\not\equiv 0$, we have that $f\not\equiv 0$. 
 
Next we show that the existence of a global harmonic function in the example for $K\leq 0$. Note that the space we are considering is a cone at infinity. We can solve the Dirichlet problem on $\partial B_i$ with almost the same boundary data to get a harmonic function $u_i:B_i\to \R$. Now it suffices to show that $u_i$ converges to a non-trivial global harmonic function $u_{\infty}:X\to \R$. The key is a three circle theorem to control the behavior of $u_i$. The observation dates back to \cite{Z,Di}, cf. \cite{X}.

 Denote $M_u(r)=\frac{1}{\mu(B_r)}\int_{B_r}u^2d\mu$.  For the space constructed above, from a similar argument as Theorem 3.2 in \cite{X}, we have that given $s$ such that \begin{align}s(s+(1-\eta))\notin\{k(k+1),k\in \N\} \end{align} there exists integer $k_0>1$ such that for any $r\geq k_0$ and $u(x)$ harmonic over $B_r$, if \begin{align}
     M_u(r)\leq 2^{2s}M_u(\frac{r}{2})\label{ttt1}
 \end{align} then \begin{align}
     M_u(\frac{r}{2})\leq 2^{2s}M_u(\frac{r}{4}).\label{ttt2}
 \end{align}
 
 Then using Lemma 10.7 in \cite{C} (see the argument of Theorem 2.1 of \cite{Di}) we have that given any harmonic function $u_\infty$ on the tangent cone at infinity of X, we can find a sequence of harmonic functions $u_i$ on $B_{R_i}$ with $R_i\to\infty$ such that \begin{align}
     \mathop{lim}\limits_{i\to\infty} |u_i\circ \Psi_{\infty,i}-u_\infty|_{L^\infty(B_\infty(1))}=0,
 \end{align}
 where $\Psi_{\infty,i}:B_\infty(1)\to B_{R_i}$ is an $\ep_i$-approximation with $\mathop{lim}\limits_{i\to\infty}\ep_i=0$.

Finally we take $\alpha$ such that $\alpha(\alpha+(1-\eta))=2$ and $\phi_1$ is a eigenfunction on $S^2$ with respect to the eigenvalue $2$. Set $u_\infty=r^\alpha\phi_1(x)$ and construct $u_i$ as above. From a blow up argument (cf. Lemma 4.4 in \cite{X}) we know that for any $d>\alpha$ with $\alpha(\alpha+(1-\eta))\notin\{k(k+1),k\in \N\}$ and $r_0\in (0,1)$ there exists $i_0=i_0(d-\alpha_1,r_0)>0$ such that if $i\geq i_0$, for any $r\in [r_0R_i,R_i]$, \begin{align}
    M_{u_i}(r)\leq 2^{2d}M_{u_i}(\frac{r}{2}).
\end{align}
 
 Combining with \eqref{ttt1}, \eqref{ttt2} and Cheng-Yau gradient estimate we have that there exists $k_0(d)$ such that \begin{align}
     \tilde{u}_i=\frac{u_i}{\sqrt{M_{u_i}(\frac{k_0}{2})}}
 \end{align} 
converge to a global non-trivial harmonic function on X. This is similar to the proof in the smooth case, cf. \cite{Di,X}.\\

\textit{Step 4. Failure of strong unique continuation.} 

Now given any harmonic function $f$ defined on an open subset of $X$ containing the vertex, without loss of generality we assume that $f(0)=0$. From Cheng-Yau gradient estimate (\cite{CY}, cf. \cite{R}) we have that on $B_r(0)$,
\begin{align}
    \sup_{B_r(0)}|\na u|\leq \frac{C}{r}\sup_{B_{2r(0)}}|u|.
\end{align}
 
 From the maximum principle and the fact that $\partial B_r(0)$ is connected, we have that on $\partial B_r(0)$ there exists $x_r\in \partial B_r(0)$ such that $u(x_r)=0$. As $diam(\partial B_r(0))=\frac{\pi}{2}r^{1+\ep}$, we have that 
 \begin{align}
     \sup_{\partial B_r(0)}|u|\leq \frac{\pi}{2}r^{1+\ep} \sup_{\partial B_r(0)}|\na u|\leq Cr^{\ep} \sup_{\partial B_{2r}(0)}|u|.
 \end{align}
 Denote
 \begin{align}
     q(r)=\sup_{\partial B_{r}(0)}|u|
 \end{align}
 and 
  \begin{align}
     D=\sup_{\frac{r_0}{2}\leq r\leq r_0}q(r).
 \end{align}
 Thus we have that 
 \begin{align}
     q(r)\leq Cr^{\ep}q(2r).
     \label{decay}
 \end{align}
 Using \eqref{decay}, we have that on $\frac{r_0}{2^{k+1}}\leq r\leq \frac{r_0}{2^k}$, $q(r)\leq 2^{\frac{\ep k(k+1)}{2}}(Cr^{\ep})^kD\leq (Cr^{\frac{\ep}{3}})^{\log_2 r_0-\log_2r}D$. Thus we have 
 \begin{align}
     q(r)\leq Cr^ke^{-\frac{\ep}{3} (\log r)^2}.\label{qest}
 \end{align}
Since \begin{align}
    \lim_{r\to 0}\frac{e^{-\frac{\ep}{3} (\log r)^2}}{r^m}=\lim_{r\to 0}e^{-\frac{\ep}{3} (\log r)^2-m\,\log r}=0,
\end{align}for any $m\geq 0$, we have that  
 \begin{align}
    \fint_{B_r(0)}f\leq Cq(r)\leq  Ce^{-\frac{\ep}{3} (\log r)^2+k\log r} =o(r^m),\,\,\text{ for all } m\geq 0.
 \end{align}
 This offers a counterexample for strong unique continuation property if we can find a harmonic function $f\not\equiv 0$, whose existence is guaranteed by the step 3.
 
In all we show that for the spaces constructed above, the strong unique continuation property of harmonic functions fails.

 \end{proof}

\begin{Rem}
 Although there is a Hopf fibration $S^1\to S^3\stackrel{\pi}{\longrightarrow} S^2$, which can be regarded as a Riemannian submersion with totally geodesic fibres and they carry the metrics $g^{S^1},g^{S^3},\frac{1}{4}g^{S^2}$, it is unclear to us whether one can construct a Ricci limit space using warped products as in \cite{CC1}. More precisely, denote $k_2=\pi^*(\frac{1}{4}g^{S^2})$ and $g^{S^3}=k_1+k_2$ and we consider metric on $\R^4\setminus\{0\}$ as
 \begin{align}
     dr^2+(\chi r^{1-\eta})^2k_1+(\frac{1}{2}r^{1+\ep})^2k_2.
 \end{align}
We have, for $X$ tangent to $S^1$ and $Y$ tangent to $S^2$,  \begin{align}
    Ric(\partial_r)&=\frac{\eta(1-\eta)+2\ep(1+\ep)}{r^2}\partial_r,\\
    Ric(X)&=\frac{\eta(1-\eta)-2(1-\eta)(1+\ep)}{r^2}X,\\
    Ric(Y)&=(\frac{2\ep(1+\ep)}{r^2}+\frac{4-(1+\ep)^2r^{2\ep}}{r^{2+2\ep}}-\frac{(1-\eta)(1+\ep)}{r^2})Y.
\end{align} 
Such spaces do not have Ricci curvature lower bound in the $S^1$ direction as $\eta<2(1+\ep)$.
\end{Rem}

\end{document}